\documentclass{amsart}
\usepackage[normalem]{ulem}
\usepackage{hyperref}
\usepackage{algorithm}

\usepackage{amsmath,amsthm}
\usepackage{color}

\DeclareMathOperator{\R}{\mathbb{R}}
\DeclareMathOperator{\Rw}{\mathbb{R}\{\!\{\varpi\}\!\}}
\DeclareMathOperator{\C}{\mathbb{C}}

\DeclareMathOperator{\K}{\mathbb{K}}
\DeclareMathOperator{\N}{\mathbb{N}}

\DeclareMathOperator{\F}{\mathbb{F}}
\DeclareMathOperator{\Fbar}{\overline{\mathbb{F}}}

\DeclareMathOperator{\I}{\mathcal{I}}
\DeclareMathOperator{\J}{\mathcal{J}}

\DeclareMathOperator{\LT}{LT}
\DeclareMathOperator{\LC}{LC}

\DeclareMathOperator{\rad}{rad}

\DeclareMathOperator{\boo}{\texttt{boo}}

\DeclareMathOperator{\rem}{rem}

\newtheorem{theorem}{Theorem}[section]
\newtheorem{lemma}[theorem]{Lemma}
\newtheorem{proposition}[theorem]{Proposition}
\newtheorem{corollary}[theorem]{Corollary}
\theoremstyle{definition}
\newtheorem{definition}[theorem]{Definition}
\newtheorem{example}[theorem]{Example}
\theoremstyle{remark}
\newtheorem{remark}[theorem]{Remark}

\numberwithin{equation}{section}
\begin{document}

\title[Deciding lower-boundedness 
of polynomials]{Deciding lower-boundedness 
of polynomials}



\author{Nguyen Hong Duc}
\address{Nguyen Hong Duc: TIMAS, Thang Long University, Vietnam}
\curraddr{}
\email{duc.nh@thanglong.edu.vn}
\thanks{}

\author{Vu Trung Hieu}
\address{Vu Trung Hieu: Center for Advanced Intelligence Project, RIKEN, Japan}
\curraddr{}
\email{trunghieu.vu@riken.jp}
\thanks{}

\subjclass[2020]{13P10, 90C23}

\date{}

\dedicatory{}


\begin{abstract}
This paper addresses the problem of deciding the lower-bound\-edness of an arbitrary real polynomial $p$ in $n$ variables. To this end, we introduce and investigate the $n$-th non-critical tangency value polynomial of 
$p$ at a point $a$ in the ambient space, denoted by $\varphi_{p,a}$. This is a univariate polynomial with coefficients in the field of Puiseux series, and its roots represent the $n$-th non-critical tangency values of $p$ at $a$. 

We demonstrate that, when $a$ is generic, the decision problem can be reduced to identifying the number of certain roots of $\varphi_{p,a}$, which can be effectively handled by Sturm’s theorem. 
Building on this approach, we design a probabilistic algorithm that decides whether a given real polynomial $p$ is lower-bounded. 
The experimental results indicate that, in the tested scenarios, the proposed algorithm is highly efficient and outperforms the sole existing method based on quantifier elimination.
\end{abstract}

\maketitle

\section{Introduction}\label{sec:intro} 
A real polynomial $p$ in $n$ variables is said to be lower-bounded over $\R^n$ if there exists a real number $\gamma$ such that all of its values are greater than $\gamma$. This property is fundamental in polynomial optimization, since it guarantees that the infimum of $p$ is finite. In fact, lower-boundedness constitutes an explicit prerequisite for algorithms, for example in \cite{schweighofer2006global,greuet2014}, that compute global infima.

In the seminal work on polynomial optimization \cite{shor1987class}, Shor emphasized the intrinsic difficulty of deciding whether a polynomial is lower-bounded, stating that
\begin{flushright}
``\textit{Checking that a given polynomial function is lower-bounded is far from
trivial}.''
\end{flushright}
This decision problem can be solved completely by using quantifier elimination over the reals. In particular, the polynomial $p$ is lower-bounded if and only if the following statement is true:
\begin{equation}\label{eq:QE}
    \exists \, t \, \forall \, x: p(x)-t>0.
\end{equation}
Determining the Boolean truth value of this statement can be achieved via the cylindrical algebraic decomposition (CAD) \cite{Collins1975}; however, its computational complexity grows doubly exponentially with the number of variables $n$ \cite{davenport1988real}. Although many refinements of CAD have been developed (see \cite{chen2016quantifier} and the references therein), solving the decision problem by means of quantifier elimination still suffers from extremely high computational complexity. Moreover, to the best of our knowledge, no approach independent of quantifier elimination has yet been proposed to address this problem.

Despite its fundamental importance in polynomial optimization, there exist only a few works addressing the theoretical characterization of lower-bound\-eness. Recently, a necessary and sufficient condition for a polynomial $p$ to be lower-bounded, formulated in terms of its tangency values, was established in \cite{pham2023tangencies}; However, its proof does not yield an explicit algorithmic framework for resolving the lower-boundedness decision problem.


In this paper, we establish a new necessary and sufficient criterion for lower-boundedness. Our approach builds upon the technique developed in \cite{duc2025nonneg} for deciding non-negativity using Sturm’s theorem. In particular, we introduce the so-called the $n$-th non-critical tangency value polynomial of 
$p$ at a point $a$ in $\R^n$, denoted by $\varphi_{p,a}$, see Definition  \ref{def:tangency_pol}. This is a univariate polynomial with coefficients in the field of Puiseux series $\F=\R\{\!\{\varpi\}\!\}$, where $\varpi$ is a positive infinitesimal. 
The roots of $\varphi_{p,a}$ correspond to the $n$-th non-critical tangency values of $p$. Our main theoretical result (Theorem \ref{thm:lower-bounded}) states that, for a generic point $a\in\R^n$, 
\begin{center}
    \textit{$p$ is lower-bounded if and only if ${v}(p,a)=0$},
\end{center}
where $v(p,a)$ is the difference of the numbers of sign variations of the Sturm sequence of $\varphi_{p,a}$ at minus infinites $-\infty_{\F}$ and $-\infty$. Consequently, the decision problem is reduced to identifying those roots of $\varphi_{p,a}$ that lie in the interval $(-\infty_{\F},-\infty)$. 

Building on this approach, we design a probabilistic algorithm (Algorithm \ref{alg:lbounded}) that decides whether a given real polynomial $p$ is lower-bounded. 
Our experiments indicate that the proposed algorithm achieves higher efficiency than using quantifier elimination for small-scale dense polynomials.

The remainder of this paper is organized as follows.
Section~\ref{sec:pre} introduces notations and auxiliary results on asymptotic tangency values, computational algebraic geometry, Puiseux series, and Sturm’s theorem.
Fundamental results concerning non-critical tangency value polynomials are established Section~\ref{sec:tangency}. 
Section~\ref{sec:Tgoodness} defines the property of being T-good, proves its genericity, and provides a procedure to test it.
Section~\ref{sec:lbounded} presents Algorithm \ref{alg:lbounded} for deciding lower-boundedness of polynomials.
Finally, Section~\ref{sec:experiment} reports the experimental results and compares the performance of the proposed algorithm with that of quantifier elimination.

\section{Background and notations}\label{sec:pre}

We denote by $R[x]$ the ring of polynomials in the variables $x=(x_1,\dots,x_n)$ with coefficients in a ring $R$.
We use $\R$ and $\C$ to denote the fields of real and complex numbers, respectively.

\subsection{Asymptotic tangency values} Let $p$  be a non-constant polynomial in $\R[x]$. 
The \textit{tangency set} \cite{pham2016genericity} of $p$ at a point $a\in \R^n$ is defined to be the following set
\[\Gamma_{\R}(p,a) :=  \big\{x \in \R^n \, | \text{ there exists } \lambda\in\R \text{ s.t. } \nabla p (x) -\lambda (x-a)= 0\big\},\]
where $\nabla p $ is the gradient vector of $p$.
The condition in the definition of $\Gamma_{\R}(p,a)$ is equivalent to the fact that the system of two vectors $\{\nabla p (x),
x-a\}$ is linear dependent. Geometrically, $\Gamma_{\R}(p,a)$ consists of all points $x$ in $\R^n$ where the level
set of $p$ is tangent to the sphere centered at $a$ with radius $\|x-a\|$. The set $\Gamma_{\R}(p,a)$ is an  unbounded semi-algebraic set \cite[Lemma 2.3]{pham2016genericity}. 
It follows from the definition that $\Gamma_{\R}(p,a)$  contains the set of critical points of $p$ - the real zeros of $\nabla p$. 

Let $\widetilde{\R} := \R \cup \{-\infty, +\infty\}$. The set of \textit{asymptotic tangency values} of $p$ at $a$ is defined as the following set: 
$$T_{\infty}(p,a)=\Big\{\lambda\in \widetilde{\R}\mid \exists x_k\in \Gamma_{\R}(p,a), \ \|x_k\|\to +\infty, \ p(x_k)\to \lambda\Big\}\subset \widetilde{\R}.$$

The set of asymptotic tangency values is finite. Moreover, $p$ is lower-bounded if and only if the minimum of $T_{\infty}(p,a)$ is finite or equal to $+\infty$. These statements are summarized in the following theorem.

\begin{theorem}[{\cite[Lemma 2.1]{vui2008lojasiewicz}, and \cite[Theorem 4.1]{pham2023tangencies}}]\label{thm:pham_tangency}
Let $a\in\R^n$ be given. Then, the set $T_{\infty}(p,a)$ is finite. Moreover, the polynomial $p$ is lower-bounded if and only if $$\min\big\{T_{\infty}(p,a)\big\}>-\infty.$$
\end{theorem}

\subsection{Ideals and Gr\"{o}bner bases} Let $\K$ be a field, and denote by $\overline{\K}$ its algebraic closure.
An additive subgroup $\I$ of $\K[x]$ is said
to be an \textit{ideal} of $\K[x]$ if $pq\in \I$ for any $p\in \I$ and
$q\in\K[x]$. Given $p_1,\dots,p_s $ in $ \K[x]$, we denote by
$\left\langle
p_1,\dots,p_s \right\rangle$ the ideal generated by $p_1,\dots,p_s$. If $\I$
is an ideal of $\K[x]$ then, according to Hilbert's basis theorem
 \cite[Chapter 2, \S 5]{cox2013},
there exist $p_1,\dots,p_s \in \K[x]$ such that
$\I=\left\langle p_1,\dots,p_s \right\rangle$.

\if
The \textit{radical} of $\I$, denoted by $\rad(\I)$, is the following set:
\[\big\{g\in\K[x]:g^k\in \I \text{ for some }k \in \N\big\}.\]
The set $\rad(\I)$ is also an ideal that contains $\I$.
One says that an ideal $\I$ is radical if $\I=\rad(\I)$.
It should be noted that the radical of any ideal is itself radical.
\fi 

The \textit{variety} of an ideal $\I$, denoted by $V_{\overline{\K}}(\I)$, is defined as the set of all common complex zeros of the polynomials in $\I$:
$$V_{\overline{\K}}(\I) = \Big\{\alpha\in\overline{\K}^n: p(\alpha) = 0, \forall p\in \I\Big\}.$$

Let $>$ be a monomial ordering on $\K[x]$
and $\I\neq \{0\}$ be an ideal of $\K[x]$. We denote
by $\LT_>(\I)$ the set of all leading terms $\LT_>(p)$ for $p\in \I$, and  by
$\left\langle \LT_>(\I)\right\rangle $ the ideal generated by the elements
of $\LT_>(\I)$. A subset $G=\{p_1,\dots,p_s\}$ of $\I$ is said to be a \textit{Gr\"{o}bner
    basis} of $\I$ with respect to the monomial order $>$ if
\[ \left\langle\LT_>(p_1),\dots,\LT_>(p_s)\right\rangle=\left\langle\LT_>(\I)\right\rangle.\]
A Gröbner basis of $\I$ really is a generating set of $\I$, i.e. $\left\langle p_1,\dots,p_s\right\rangle=\I$. Every ideal in $\K[x]$ has a Gr\"{o}bner basis. A Gr\"{o}bner basis $G$ is \textit{reduced} if the two following conditions hold: the leading coefficient of $p$ is $1$, for all $p\in G$; and, for all $p\in G$, there are no monomials of $p$ lying in $\left\langle \LT_>(G\setminus\{p\})\right\rangle$. Every ideal $\I$ has a unique reduced Gr\"{o}bner basis. Algorithms for computing Gröbner bases, such as those presented in \cite{buchberger2006,faugere2002new}, are well established. For additional background, we refer the reader to \cite{cox2013}. 

\subsection{The field of Puiseux series over $\R$}\label{subsec:signchanges} Denote by $\K\{\!\{\varpi\}\!\}$ the \textit{field of Puiseux series} over a field $\K$,  where $\varpi$ is a positive infinitesimal. Recall  that, $\K\{\!\{\varpi\}\!\}$ 
consists of the series of form
$$\varphi(\varpi)=\sum_{n\geq \ell}a_n \varpi^{n/N}\ \text{ with } n, N\in \Bbb Z, a_n\in \Bbb K, a_{\ell}\neq 0, N>0.$$
The exponent $\ell/N$ is called the order of $\varphi$, denoted by $\mathrm{ord}\ \varphi$. The term $a_{\ell}\varpi^{\ell/N}$ and $a_{\ell}$ are called the leading term and leading coefficient of $\varphi$ and denoted by $\LT(\varphi)$ and $\LC(\varphi)$ respectively. There is a morphism 
$$\lim_{\varpi\to 0}\colon \R\{\!\{\varpi\}\!\}\to \R\cup\{-\infty,+\infty\}$$ satisfying
$$
\lim_{\varpi\to 0}\varphi(\varpi)=
\begin{cases}
0, \text{ if } \mathrm{ord}\ \varphi>0,\\
\LC(\varphi), \text{ if } \mathrm{ord}\ \varphi=0,\\
+\infty \text{ if } \mathrm{ord}\ \varphi<0 \text{ and } \LC(\varphi)>0,\\
-\infty \text{ if } \mathrm{ord}\ \varphi<0 \text{ and } \LC(\varphi)<0.
\end{cases}
$$
In $\Rw$ we consider a total order $>$ defined as follows:
 $\varphi>0$ if and only if $\LC(\varphi)>0$.
Then, the field $(\Rw,>)$ is a real closed field. That is, every positive element of $\Rw$ has a square root in $\Rw$ and any polynomial of odd degree with coefficients in $\Rw$ has at least one root in $\Rw$.  The algebraic closure of $\Rw$ is isomorphic to $\C\{\!\{\varpi\}\!\}$. For further details, we refer the reader to \cite[Chapter 2]{basu2006algorithms}.

Throughout this paper, we denote by $\F$ the field $\Rw$ and by $\Fbar$ its algebraic closure. We also use $\R(\varpi)$ to denote the field of rational functions in $\varpi$ over $\R$. 

In $\F$, one may define a notion of limit analogous to the usual epsilon-delta ($\varepsilon$-$\delta$) definition in $\R$, denoted by $\F\!\lim$, where $\varepsilon,\delta$ are assumed to be in $\F$. 
Let us denote 
$$-\infty_{\F}={\F\!\lim}_{n\to +\infty} - \varpi^{-n},\ +\infty_{\F}={\F\!\lim}_{n\to +\infty}  \varpi^{-n}$$ and $\widetilde{\mathbb{F}}:=\mathbb{F}\cup \{-\infty_{\F},+\infty_{\F}\}$. Note that the sequence $\{n\}$ is not convergent in $\tilde\F$.   

\subsection{Sturm's theorem}\label{thm:sturm} 
Let $\varphi$ be a polynomial in $\F[t]$.     The Sturm sequence of the polynomial $ \varphi $ is defined as follows: 
    \( \varphi_0 := \varphi \), \( \varphi_1 := \varphi' \), 
    and for each \( i > 1 \), 
    \[ \varphi_{i+1}:=-\rem(\varphi_{i-1},\varphi_{i}),\] 
    where $\rem(\varphi_{i-1},\varphi_{i})$ is the remainder of the Euclidean division of \( \varphi_{i-1} \) by \( \varphi_i \). It is clear that
    \(\deg(\varphi_{i+1}) < \deg(\varphi_i).
    \)
Let $k$ be the first index $i$ such that $\deg \varphi_{i} =0$.  The last nonzero polynomial \( \varphi\) is then a nonzero constant.  The length $k$ of the Sturm sequence is at most the degree of $\varphi$.

Let \( \alpha \in \mathbb{F}\) be a point. Define \( v_{\varphi}(\alpha) \) as the number of sign changes in the sequence \( \varphi_0(\alpha), \varphi_1(\alpha), \ldots, \varphi_k(\alpha) \), ignoring any zero values. The following result is Sturm's theorem over the real closed field $\mathbb{F}$, see \cite[Theorem 2.50]{basu2006algorithms}.

\begin{theorem}[Sturm's theorem]\label{sturmtheorem}
Let \( \alpha < \beta \) be elements in $\F\cup\{-\infty_{\F},-\infty\}$ such that neither $\alpha$ nor $\beta$ are not roots of $\varphi$. Then the number of roots of \( \varphi \) in the interval \( (\alpha, \beta)\) in $\F$ is given by:
\[
v_{\varphi}(\alpha) - v_{\varphi}(\beta).
\]
\end{theorem}

\section{Properties of non-critical tangency value polynomials}\label{sec:tangency}
Let $p$  be a non-constant polynomial in $\R[x]$. This section presents the concept of the $n$-th non-critical tangency value polynomial of $p$ at a given point $a\in\Fbar^n$, where $\Fbar$ is the algebraic closure of $\F=\Rw$, 
and discusses its main properties. 

\subsection{Tangency points and values over $\Fbar$}
The following $n$ polynomials, associated with the partial derivatives of $p$ and the point $a$, will be used frequently throughout this work:
\begin{equation}\label{eq:Phi_i}
\Phi_i(x,a)=(x_i-a_i){\partial p}/{\partial x_n}-(x_n-a_n){\partial p}/{\partial x_i}, \ i=1,\dots,n-1,
\end{equation}
and
\begin{equation}\label{eq:Phi_n}
\Phi_n(x,a):=\varpi \big((x_1-a_1)^2+\ldots+(x_n-a_n)^2\big)-1.
\end{equation}

The set of \textit{tangency points} over $\Fbar$ 
of $p$ at $a\in\Fbar^n$  is defined as the following subset of $\Fbar^n$:
\begin{equation}\label{eq:Gamma_ainR}
\Gamma_{\Fbar}(p,a) :=  \big\{x \in\Fbar^{n} 
\mid \Phi_i(x,a)=0,\ \forall i=1,\cdots,n\big\}.
\end{equation}


Apart from $\Gamma_{\Fbar}(p,a)$, we now define a subset of it, which will play a crucial role in the subsequent analysis.
\begin{definition}
The following set is called the \textit{$n$-th set of non-critical tangency points} of $p$ with respect to $a$:
\[\Delta_{\Fbar}(p, a):=\Gamma_{\Fbar}(p, a)\setminus \Sigma^n_{\C}(p),\]
where $\Sigma^n_{\C}(p)$ is the set of complex roots of ${\partial p}/{\partial x_n}$ -- the $n$-partial derivative of $p$. 
\end{definition}

The $i$-th set of non-critical tangency points can be defined analogously by replacing $n$ with $i$, where $i = 1, \dots, n-1$. In this study, it suffices to consider only the $n$-th set. Since the context is clear, the index $n$ is be omitted from the notation. 

Clearly, $\Delta_{\Fbar}(p, a)$ does not contain the set of complex critical points of $p$. 
This set can be expressed as the solution set of the following mixed system consisting of 
$n$ polynomial equations and one polynomial inequation:
\begin{equation}\label{eq:Delta}
\Delta_{\Fbar}(p,a) =  \big\{x \in\Fbar^{n} 
\mid  \Phi_1(x,a)=0, \cdots, \Phi_{n}(x,a)=0, {\partial p}/{\partial x_n}\neq 0\big\}.
\end{equation}



\begin{proposition}\label{prop:finite}
The image set $p\left(\Delta_{\Fbar }(p, a)\right)\subset \Fbar$ is finite.
\end{proposition}
\begin{proof} Since $\Delta_{\Fbar}(p,a)\subset \Gamma_{\Fbar}(p,a)$, the finiteness of $p(\Delta_{\Fbar}(p,a))$ is an immediate consequence of the finiteness of $p(\Gamma_{\Fbar}(p,a))$. Thus, it remains to establish the finiteness of $p(\Delta_{\Fbar}(p,a))$.

We consider $\varpi$ as a non-zero real parameter. Let $\I_{\varpi}$ be the ideal $\I$ in $\R[x]$ defined as follows:
\begin{equation}\label{eq:Iw}
    \I_{\varpi}:=\big\langle \Phi_1(x,a), \, \dots, \, \Phi_{n-1}(x,a), \, \Phi_n(x,a) \big\rangle.
\end{equation}
From the definitions of $\Phi_i$'s in \eqref{eq:Phi_i} and \eqref{eq:Phi_n}, $V_{\C}(\I_{{\varpi}})$ is the set of complex critical points of $p$ over the sphere centered at $a$, which is considered as a complex vector with parameter $\varpi$, with radius $1/w$. This set has finitely connected components, and according to the proof of
\cite[Proposition 3.2]{duc2025nonneg}, $p$ is constant in each of its connected component. Hence, the image set $p(V_{\C}(\I_{{\varpi}}))$ is finite (in $\Fbar$). It follows from \eqref{eq:Gamma_ainR} that if we substitute $ V_{\C}(\I_{\varpi})=\Gamma_{\Fbar}(p,a)$, then the set $p(\Gamma_{\Fbar}(p,a))\subset \Fbar$ is finite. 
\end{proof}

\begin{remark}
It follows from the proof of Proposition \ref{prop:finite} that
the image set $p\left(\Gamma_{\Fbar }(p, a)\right)$ is finite.
\end{remark}

\begin{example}\label{ex:x_y0}
Consider the bivariate polynomial $p=x_1-x_2$ and $a=(0,0)$. Since $\partial p/\partial x_2=-2$, $\Gamma_{\Fbar }(p, a)$ coincides with $\Delta_{\Fbar }(p, a)$, and  \eqref{eq:Gamma_ainR} becomes
\[\begin{array}{cl}
\Gamma_{\Fbar}(p,a)  & = \  \{x \in\Fbar^{n} \mid x_1+x_2=0, w (x_1^2+x_2^2)-1=0\} \medskip \\
     &= \ \{x \in\Fbar^{n} \mid x_1=-x_2=\pm\frac{1}{\sqrt{2}}\varpi^{-1/2}\}.
\end{array}\]
Therefore, $p\left(\Gamma_{\Fbar }(p, a)\right)=\{\pm\sqrt{2}\varpi^{-1/2}\}$ is finite.
\end{example}

In the following, we introduce some notations concerning the sets of tangency points in $\F^n$ and $\Fbar^n$:
\[\Gamma_{\F}(p, a):=\F^n\cap\,\Gamma_{\Fbar}(p, a) \text{ and } \, \Delta_{\F}(p, a):=\F^n\cap\,\Delta_{\Fbar}(p, a).\]

The following property clarifies the importance of $\Delta_{\F}(p, a)$, and hence of $\Delta_{\Fbar}(p, a)$, in analyzing the behavior of $p$ at $-\infty$.

\begin{lemma}\label{lm:infty} Let $x$ be in $\Gamma_{\F}(p, a)$. If
    $\lim_{\varpi\to 0} p(x)=-\infty$, then $x$ lies in $\Delta_{\F}(p, a)$.
\end{lemma}
\begin{proof}
Assume that $\lim_{\varpi \to 0} p(x) = -\infty$, but, on the contrary, $x \notin \Delta_{\F}(p, a)$.
It then follows from the definition of $\Delta_{\F}(p, a)$ that $x \in \Sigma^n_{\R}(p)$, and hence $p(x)$ is a real number independent of $\varpi$.
This leads to a contradiction.
\end{proof}

We now discuss the relationship of the set of asymptotic tangency values and $\Gamma_{\F}(p, a)$ - the set of tangency points in $\F$.
 
\begin{lemma}
    \label{prop:asym_tangency_values}
Let $a$ be a given point in $\R^n$. Then, the following identity holds
$$T_{\infty}(p,a)=\Big\{\lim_{\varpi\to 0} p(x)\mid x\in \Gamma_{\F}(p, a)\Big\}.$$
\end{lemma}
\begin{proof}
We only need to prove the inclusion $\subseteq$,
as the inverse is obvious. Assume that $\lambda\in T_{\infty}(p,a)$. Then there exists a sequence $x_k\in \Gamma_{\F}(p, a)$ such that $\|x_k\|\to\infty$ and $p(x_k)\to \lambda$. By a version at infinity of the curve selection lemma \cite{nemethi1992milnor}, there exists a Nash curve $\varphi(\tau)$ in $\R^n$ of form
$$\varphi(\tau)=\sum_{k\leq k_0} a_k \tau^k \text{ with } a_k\in \R^n, a_{k_0}\neq 0,$$
such that $\Phi_i(\varphi(\tau))=0,$ for all $i=1,\ldots,n$ and 
$ \varphi(\tau)\to \infty$ and $p(\varphi(\tau))\to \lambda $  as $\tau\to \infty$.
We see that $k_0>0$ because $\varphi(\tau)\to \infty \text{ as }\tau\to \infty$.
Solving the equation $\varpi=\|\varphi(\tau)-a\|^2$, we obtain a series $\tau=\phi(\varpi)$ in $\F$ satisfying $$\varpi=\|\varphi(\phi(\varpi))-a\|^2 \text{ and } \lim_{\varpi\to 0} \|\phi(\varpi)\|=+\infty.$$ Then, the element $x=\varphi(\phi(\varpi))$ lies in $\Gamma_{\F}(p, a)$ and satisfies $\lim_{\varpi\to 0} p(x)=\lambda$.
\end{proof}

The relationship of $\Delta_{\F}(p, a)$ and the lower-boundedness of $p$ is stated in the following proposition, which will be used in the proof of the main result -- Theorem~\ref{thm:lower-bounded}.
 
\begin{proposition}
    \label{prop:tangency_ord}
Let $a$ be in $\R^n$. The following statements are equivalent:
\begin{itemize}
    \item[\rm(i)] $p$ is lower-bounded;
     \item[\rm(ii)] For all $x\in \Delta_{\F}(p, a)$, the limit $\lim_{\varpi\to 0} p(x)$ belongs to $\R\cup\{+\infty\}$.
\end{itemize}
\end{proposition}
\begin{proof} Assume that $\rm(i)$ holds. 
It follows from Theorem \ref{thm:pham_tangency} and Lemma \ref{prop:asym_tangency_values} that $\rm(i)$ is equivalent to the following statement:
$\lim_{\varpi\to 0} p(x)$ belongs to $\R\cup\{+\infty\}$ for all $x\in \Gamma_{\F}(p, a)$. Consequently, $\rm(i)$ implies $\rm(ii)$.

It suffices to show that $\rm(ii)$ implies $\rm(i)$. Indeed, assume that $\rm(ii)$ holds, but, on the contrary, $-\infty\in T_{\infty}(p,a)$. There exist $x \in \Gamma_{\F}(p, a)$ such that  $\lim_{\varpi\to 0} p(x)=-\infty$. According to Lemma \ref{lm:infty}, $x$ belongs to $\Delta_{\F}(p, a)$. This yields a contradiction.
\end{proof}

\subsection{Non-critical tangency value polynomials}
Recall that a polynomial $\varphi$ in $\R(\varpi)[t]$ 
is square-free if $\gcd(\varphi,\varphi')=1$, i.e.  the greatest common divisor of $\varphi$ and its derivative is $1$. A polynomial $\theta$ in $\R(\varpi)[t]$ is the square-free part of 
$\varphi$ if $\varphi=\theta\times \gcd(\varphi,\varphi')$.

By Proposition~\ref{prop:finite}, one can define a univariate polynomial in $\Fbar[t]$ whose roots coincide with the set $p(\Delta_{\Fbar}(p,a))$. The following proposition strengthens this result by showing that the coefficients of such a polynomial actually lie in $\R(\varpi)$.

\begin{proposition}\label{prop:exist} Let $a$ be in $\R^n$. There exists a unique, up to multiplication by a nonzero scalar in $\R(\varpi)$, square-free univariate polynomial in $\R(\varpi)[t]$ whose root set coincides with the image set $p(\Delta_{\Fbar}(p,a))$.
\end{proposition} 

\begin{definition}\label{def:tangency_pol} 
The square-free univariate polynomial in $\R(\varpi)[t]$ mentioned in Proposition \ref{prop:exist} is called
the \textit{$n$-th non-critical tangency value polynomial} of $p$ at $a$.
\end{definition}

\textit{Proof of Proposition \ref{prop:exist}} We enlarge the polynomial ring by introducing auxiliary variables $t$ and $s$. 
Let $\I$ denote the ideal in $\R(\varpi)[s,x,t]$ defined as follows:
\begin{equation}\label{eq:Idelta}
    \I:=\big\langle p(x)-t, \, \Phi_1(x,a), \, \dots, \, \Phi_n(x,a), \, s\times\partial p/\partial x_n -1 \big\rangle.
\end{equation}
According to Hilbert’s basis theorem, there exists a finite list $B$ of generators of the ideal $\I\cap \R(\varpi)[t]$. 
Since $\R(\varpi)[t]$ is a principal ideal domain, $ \langle B \rangle$ is principal, there exists a polynomial $\bar\varphi \in \mathbb{R}(\varpi)[t]$ such that 
$\langle \bar\varphi\rangle=\I\cap \R(\varpi)[t]$.
Let $\varphi$ be the square-free part of $\bar{\varphi}$. 

We prove that the root set of $\varphi$
coincides with $p(\Delta_{\Fbar}(p,a))$.
Assume that $\lambda\in\Fbar$ is a root of $\varphi$. By the definition of $\varphi$, $\lambda$ is also a root of $\bar{\varphi}$. 
Then, there exist $s\in \Fbar$ and $\bar x\in\Fbar^{n}$ such that $\lambda=p(\bar x)$ and $\bar x$ is a root of the following polynomial system
\[\big\{\Phi_1(x,a), \, \dots, \, \Phi_n(x,a), \, s\times\partial p/\partial x_n -1 \big\}.\]
The identity $s\times\partial p/\partial x_n (\bar x) -1=0$ implies that $\partial p/\partial x_n(\bar x)\neq 0$. It follows that $\bar x\in\Delta_{\Fbar}(p,a)$, and that $\lambda$ is an element of $p(\Delta_{\Fbar}(p,a))$. 
The converse inclusion, namely that each value in $p(\Delta_{\Fbar}(p,a))$ is a root of $\varphi$, is established by an analogous argument. \qed


\begin{remark} The preceding notion can be naturally extended to the case $a$ in $\F^n$. Repeating the proof of Proposition \ref{prop:exist} shows that there exists a unique square-free polynomial in $\F[t]$, up to multiplication by a nonzero scalar in $\F$, whose root set is $p(\Delta_{\Fbar}(p,a))$. This polynomial is also denoted by $\varphi_{p,a}$.
\end{remark}


\subsection{Computing $\varphi_{p,a}$}
Guided by the arguments in the proof of Proposition~\ref{prop:exist}, we design Algorithm~\ref{alg:compute_phi_pa} for computing
the polynomial $\varphi_{p,a}$ using Gr\"{o}bner basis computations.



\begin{algorithm}\caption{Computing $\varphi_{p,a}$ -- the $n$-th non-critical tangency value polynomial}\label{alg:compute_phi_pa}
    \smallskip
\begin{flushleft}
    \textbf{Input:} $p$ in $\R[x]$ and $a\in\R^n$
\smallskip

    \textbf{Output:} $\varphi_{p,a}$
\end{flushleft}
    \begin{itemize}
\item [\rm 1:] Compute $G$ -- the reduced Gr\"{o}bner basis of $\I$ given in \eqref{eq:Idelta} w.r.t the pure lexicographic monomial order  $s>x> t$ in $\R(\varpi)[s,x,t]$
\item [\rm 2:] Let $\{\bar{\varphi}\}=G \cap \R(\varpi)[t]$ 
and compute $\varphi$ -- the square-free part of $\bar{\varphi}$
\item [\rm 3:] Return $\varphi$
    \end{itemize}
\end{algorithm}

\textit{Description}.
The input of Algorithm \ref{alg:compute_phi_pa} includes
$p$ in $\R[x]$ and $a\in\R^n$. The output is the $n$-th non-critical tangency value polynomial of $p$ at $a$. Step~1 computes the the reduced Gr\"{o}bner basis $G$ of $\I$, w.r.t the pure
    lexicographic monomial order  $>$, in the ring $\R(\varpi)[s,x,t]$. 
    Step~2 guarantees that the resulting polynomial 
$\varphi$ is square-free.

\textit{Correctness}. The correctness of Algorithm~\ref{alg:compute_phi_pa} is guaranteed by the arguments in the proof of Proposition~\ref{prop:exist}: After Step 2, $\varphi$ is a square-free univariate polynomial in $\R(\varpi)[t]$ whose root set coincides with the image set $p(\Delta_{\Fbar}(p,a))$.

\begin{example}\label{ex:xy_1}
Consider the bivarite polynomial $q=(x_1x_2-1)^2 + x_2^2$ and $a=(0,0)$. One computes the reduced Gr\"{o}bner basis $G$ of the following ideal with respect to the pure
lexicographic monomial order $s>x_1>x_2>t$ in the ring $\R(\varpi)[s,x_1,x_2,t]$:
$$\I=\big\langle q -t , x_1{\partial q}/{\partial x_2}-x_2{\partial q}/{\partial x_1}, \varpi (x_1^2+x_2^2)-1, s\times{\partial q}/{\partial x_2}-1\big\rangle,$$
where $t,s$ are new variables as in Step 1. One then obtains  $\bar{\varphi}$ in Step 2 as follows:
\[\begin{array}{r}
    \bar{\varphi} \ = \ 16\varpi^5t^4-\big(40\varpi^5  + 32 \varpi^4  + 8 \varpi^3\big)t^3 +
\big(49 \varpi^5  + 60 \varpi^4  + 14 \varpi^3  + 12 \varpi^2  + \varpi\big)t^2  \\
     - \big(42 \varpi^5 + 49 \varpi^4  + 26 \varpi^3  - 2 \varpi^2  + 4 w + 1\big)t 
+ \big( 17 \varpi^5  + 21 \varpi^4  - 5 \varpi^3  + 3 \varpi^2  + \varpi\big). \
\end{array}\]
This polynomial is square-free. It follows that $\varphi\equiv\bar{\varphi}$ and this is the tangency value polynomial of $q$ at $(0,0)$.
\end{example}

\section{The property of being T-good}\label{sec:Tgoodness}

\subsection{Definition and properties} This subsection introduces a new notion concerning the relationship between the cardinalities of   $p\left(\Delta_{\Fbar }(p, a)\right)$ and $\Delta_{\Fbar }(p, a)$. 

\begin{definition}\label{T-good}
    An element $a\in \Fbar ^n$ is called \emph{T-good} for the polynomial $p$ if 
\begin{equation}\label{eq:def_Tgood}
    \sharp\ p\left(\Delta_{\Fbar }(p, a)\right)=\sharp\ \Delta_{\Fbar }(p, a).
\end{equation}
\end{definition}

As the set $p\left(\Delta_{\Fbar }(p, a)\right)$ is finite, 
the T-goodness of $a$ implies that the set $\Delta_{\Fbar }(p, a)$ is finite as well, and that $p$ is injective when restricted to $\Delta_{\Fbar }(p, a)$. The genericity of the T-good property will be demonstrated in Theorem \ref{thm:generic-tangency}.

\begin{example}\label{ex:x_y1}
Consider the polynomial $p$ and point $a$ given in Example \ref{ex:x_y0}. It follows from the results of that example that
the two sets $p\left(\Gamma_{\Fbar }(p, a)\right)$ and $\Gamma_{\Fbar }(p, a)$ have the same number of elements. Hence, we conclude that $a$ is T-good for $p$. 
\end{example}

The following proposition establishes a fundamental property of T-goodness: the image of $\Gamma_{\F}(p,a)$ under $p$ coincides with the $n$-th set of non-critical tangency values of $p$ in $\F$.

\begin{proposition}\label{prop:real_values_F} Assume that $a\in \Fbar^n$ is T-good for the polynomial $p\in\R[x]$. The following equality holds:
\begin{equation}\label{eq:important}
p\left(\Delta_{\F}(p, a)\right)=p\left(\Delta_{\Fbar }(p, a)\right)\cap \F.
\end{equation}
\end{proposition}
\begin{proof}
We only need to verify the inclusion $p\left(\Delta_{\F}(p, a)\right) \supset p\left(\Delta_{\Fbar }(p, a)\right)\cap \F.$
Take any element $\lambda$ in $p\left(\Delta_{\Fbar }(p, a)\right)\cap \F$. Then, there exists $z\in \Fbar ^n$ such that  $p(z)=\lambda$ and
$z\in \Delta_{\Fbar }(p, a)$. 
However, we have
$\bar z\in \Delta_{\Fbar }(p, a)$, where $\bar z=(\bar z_1,\dots,\bar z_n)$ with $\bar z_i$ is the conjugate of $z_i$ in $\Fbar$, and $p(\bar z)=\lambda.$
 It follows from Definition~\ref{T-good} that $z=\bar z$ and hence $\lambda$ must belong to $p\left(\Delta_{\F}(p, a)\right)$. 
\end{proof} 



\subsection{Genericity of the T-goodness}
This subsection is devoted to proving the genericity of the T-goodness. For this purpose, we introduce some additional notations. 

Let $\Delta_{\Fbar}(p)$ be a subset of $\Fbar^{2n}$ defined as follows: 
\[\Delta_{\Fbar}(p) :=  \big\{(x,a) \in\Fbar^{2n} \mid \Phi_1(x,a)=0, \cdots, \Phi_{n}(x,a)=0, {\partial p}/{\partial x_n}\neq 0\big\},\]
where $\Phi_i$'s are defined in \eqref{eq:Phi_i} and \eqref{eq:Phi_n}. The two natural projections
\[\pi_1,\pi_2\colon \Fbar^{2n}\to \Fbar^n\]
are defined by $\pi_1(x,a)=x$ and $\pi_2(x,a)=a$. For each $a\in \Fbar^n$, we define a set $\Delta_{\F}(p, a)$ in $\Fbar^{n}$ as 
$$\Delta_{\Fbar}(p, a):=\pi_1\big(\pi_2^{-1}(a) \cap \Delta_{\Fbar}(p) \big).$$
Clearly, if we fix $a\in \Fbar^n$, then $\Delta_{\Fbar}(p, a)$ coincides with the set defined in \eqref{eq:Gamma_ainR}.

\begin{theorem}\label{thm:generic-tangency}
There exists a Zariski open subset $U$ in $\R^n$ such that all elements $a$ in $U$ are T-good for $p$.  
\end{theorem}
\begin{proof}
The proof is conducted in five claims as follows.

{\bf Claim 1:} \emph{The restricted morphism $\pi_2\colon \Delta_{\Fbar }(p) \to \Fbar ^n$ is Zariski dominant.} 
That is,   
the Zariski closure of $\pi_2\left(\Delta_{\Fbar }(p)\right) $ is equal to $\Fbar ^n$.
Indeed, we first prove that the map of rings 
$$\pi_2^*\colon \F[a]\to \F[x,a]/\langle\Phi_1,\ldots,\Phi_n\rangle$$
is injective. Suppose that this is not the case. Then, for some non-zero element $f\in \F[a]$, $\pi_2^*(f)=0$. This means that $$f(a)=\sum_{i=1}^n c_i(x,a) \Phi_i(x,a)$$
for some polynomials $c_i(x,a)$ in $\F[x,a]$. Let $\alpha\in \R^n$ be an element such that $f(\alpha)\neq 0$, and let $\varpi$ be fixed with a small value, then  
$c_i(x,\alpha)$  and $\Phi_i(x,\alpha)$ are polynomials in $\R[x]$ for all $i=1,\ldots,n$. Clearly, the set
\begin{equation}\label{eq:crit_Sw}
    \big\{x\in \R^n\mid \Phi_1(x,\alpha)=\cdots= \Phi_n(\varpi,x,\alpha)=0\big\}
\end{equation}
is 
the set of critical points of $p$ on the sphere centered at $\alpha$ with radius $\varpi$:
$$\mathbb{S}_{\varpi}:=\big\{x\in \R^n\mid \Phi_n(x,\alpha)=\varpi \left((x_1-\alpha_1)^2+\ldots+(x_n-\alpha_n)^2\right)-1=0\big\}.$$
The set in \eqref{eq:crit_Sw} is obviously nonempty because $\mathbb{S}_{\varpi}$ is a nonempty, compact, and smooth manifold. Let $\hat x\in \mathbb{S}_{\varpi}$ be a critical point of $p$ over $\mathbb{S}_{\varpi}$. Then $\Phi_i(\hat x,\alpha)=0$ for all $i=1,\ldots,n$. Therefore,
$$0\neq f(\alpha)=\sum_{i=1}^n c_i(\hat x,\alpha) \Phi_i(\hat x,\alpha)=0,$$
which is a contradiction. Hence, the map $\pi_2^*\colon \F[a]\to \F[x,a]/\langle\Phi_1,\ldots,\Phi_n\rangle$ is injective and so is the map 
$$\pi_2^*\colon \Fbar [a]\to \Fbar [x,a]/\langle\Phi_1,\ldots,\Phi_n\rangle.$$
This implies that the morphism $\pi_2\colon \Delta_{\Fbar }(p) \to \Fbar ^n$ is dominant.

{\bf Claim 2:} \emph{$\Delta_{\Fbar }(p)$ is a smooth variety of dimension $n$ in $\Fbar ^{2n}$.} 
Given $\alpha\in\R^n$ as shown in Claim 1, the set defined by equation  \eqref{eq:crit_Sw} is nonempty; therefore, so is $\Delta_{\Fbar }(p)$. Let $\left(\hat x, \hat a\right)$ be a point in $\Delta_{\Fbar }(p)$. By the definition, $\partial p/\partial x_n\left(\hat x\right) \neq 0$, hence there exists a neighbourhood $\mathcal V$ of $(\hat x,\hat a)$ in $\Fbar ^{2n}$ such that $\partial p/\partial x_n(x) \neq 0$ for all $(x,a) \in \mathcal V\cap \Delta_{\Fbar }(p)$. Consequently, we may write
$$
\mathcal V\cap \Delta_{\Fbar }(p)=\big\{(x, a) \in\mathcal V  \mid \Phi_i(x, a)=0, \, i=1, \ldots, n\big\}.
$$
We denote by $M(x,a)$ the following matrix:
\[M(x,a)=\left[\frac{\partial \Phi_i}{\partial a_j}(x, a)\right]_{1 \leq i, j \leq n}.\]
A direct computation shows that
$$
\varpi (x_n-a_n)\det\big(M(x,a)\big)=\big(\partial p/\partial x_n(x)\big)^{n-1} \neq 0,
$$
hence, that $\det\left(M(x,a)\right)\neq 0$
 for all $(x, a) \in \mathcal V\cap \Delta_{\Fbar }(p)$. Therefore, $\Delta_{\Fbar }(p)$ is a smooth variety of dimension $n$ in $\Fbar^{2n}$.

{\bf Claim 3:} \emph{ There exists an open analytic subset $\mathcal U$ in $\Fbar ^n$ such that all elements $a$ in $\mathcal U$ are T-good for $p$.} 

Recall from Claim 1 that $\pi_2:  \Delta_{\Fbar }(p) \rightarrow \Fbar ^n,(x, a) \mapsto a$ is dominant. By the generic smoothness theorem, see e.g.  
\cite[Theorem 25.3.3]{vakil2025rising}, there exists an open dense $\mathcal V \subset \mathbb{F}^n$ such that $\pi_2$ is smooth over $\mathcal V$ and for every $a\in \mathcal V$ the fiber $\pi_2^{-1}(a)$ is finite. Moreover, by shrinking $\mathcal V$, it follows from \cite[Proposition 9.7.8]{grothendieckEGA3} that
the number of points of $\pi_2^{-1}(a)$ is constant for all $a$ in $\mathcal V$. We now take a point $\hat a$ in $\mathcal V$ and let 
$$\pi_2^{-1}(\hat a)=\{(x^1,\hat a), \ldots,(x^r,\hat a)\}.$$
Since, for each $k=1,\ldots,r$ the morphism $\pi_2$ is smooth at $(x^j,\hat a)$, it follows that the matrix $M(x^k,\hat a)$ is invertible. According to the implicit function theorem, see e.g. \cite[Theorem I.1.18]{greuel2007introduction}, there exist an analytic open neighbourhood $\mathcal U'$ of $\hat a$ in $\mathcal V$ and neighbourhoods $\mathcal V_k$ of $x^k$ and analytic functions $\varphi_k\colon \mathcal U'\to \mathcal V_k$ such that $\varphi_k(\hat a)=x^k$ and
\[\Phi_1(\varphi_k(a),a)= \cdots =\Phi_n(\varphi_k(a),a)=0,\]
for all $a\in  \mathcal U'$. By shrinking $ \mathcal U'$ and $ \mathcal V_k$ we may assume that $\mathcal V_i\cap \mathcal V_j=\emptyset$ for all $i\neq j$. Then, one has
$$\pi_2^{-1}(\mathcal U')=\bigcup_{j=1}^r\mathcal V_j\times  \mathcal U'.$$
Indeed, take an element $a\in \mathcal U'$, the points $(\varphi_k(a),a)\in \pi_2^{-1}(a)$ for all $k=1,\ldots,r$. Hence
$$\pi_2^{-1}(a)=\{(\varphi_1(a),a),\ldots,(\varphi_r(a),a)\}$$
since $\sharp\ \pi_2^{-1}(a)=r$.

Let $\mathcal U$ be the open analytic subset in $\Fbar ^n$ defined as the set of elements $a$ in $\mathcal U'$ such that $p(\varphi_i(a))\neq p(\varphi_j(a))$ for all $i\neq j $ in $\{1,\ldots,r\}$. Then
$$\sharp\ p\left(\Delta_{\Fbar }(p, a)\right)=\sharp\ \Delta_{\Fbar }(p, a)=r,$$
for all $a\in \mathcal U$.

{\bf Claim 4:} \emph{ There exists a Zariski open subset $\mathcal U$ in $\Fbar ^n$ such that all elements $a$ in $\mathcal U$ are T-good for $p$.}

Consider the morphism $\tilde{p}\colon \Fbar ^{2n} \to \Fbar ^{n+1}$ mapping $(x,a)$ to $(p(x),a)$ and denote by $Y$ the Zariski closure of $\tilde{p}\left(\Delta_{\Fbar }(p)\right)$ in $ \Fbar ^{n+1}$. Applying \cite[Proposition 9.7.8]{grothendieckEGA3}
to the morphisms $\pi_2$ and $$\tilde{\pi}_2\colon Y\to  \Fbar ^{n}, (t,a)\mapsto a,$$ 
we obtain a Zariski open subset $\mathcal U$ of $\Fbar ^{n}$ such that $\sharp\ \pi_2^{-1}(a)$ and $\sharp\ \tilde{\pi}_2^{-1}(a)$ are constant on $\mathcal U$. By Claim 3, there exists an analytic open subset $\mathcal U'\subset \Fbar ^{n}$ such that all $a\in \mathcal U'$ are T-good for $p$. As $\mathcal U$ is Zariski open in $\Fbar ^{n}$, the intersection $\mathcal U\cap \mathcal U'$ is nonempty. 
Hence, for all $a\in \mathcal U$, we have
\begin{align*}
 \sharp\ p\left(\Delta_{\Fbar }(p, a)\right)= \sharp\ p\left(\Delta_{\Fbar }(p, b)\right)=\sharp\ \Delta_{\Fbar }(p, b)=\sharp\ \Delta_{\Fbar }(p, a),
\end{align*}
where $b$ is an element in $\mathcal U\cap \mathcal U'$.

{\bf Claim 5:} \emph{There exists a Zariski open subset $U$ in $\R^n$ such that all elements $a$ in $U$ are T-good for $p$.}

We consider the ideal 
$$\I=\langle\Phi_1(x,a),\ldots,\Phi_n(x,a), p(x)-t\rangle$$ in $\R[\varpi,x,a,t]$.
Let $\varphi$ denote a (square-free) polynomial in $\R[\varpi,a,t]$ generating the radical of the ideal $\I\cap \R(\varpi,a)[t]$.

Let
$$\J=\langle\Phi_1(x,a),\ldots,\Phi_n(x,a), g(x,b)-t\rangle.$$
be an ideal in $\R[\varpi,x,a,b,t]$, where $g(x,b)=b_1x_1+\ldots+ b_nx_n$. Let $\psi$ be a square-free polynomial in $\R[\varpi,a,b,t]$ generating the radical of the ideal $\J\cap \R(\varpi,a,b)[t].$
Notice that, for each $a\in \Fbar^n$ one has
$$\deg\varphi_{p,a}(t)=\sharp\ p\left(\Delta_{\Fbar }(p, a)\right) \text{ and } 
\deg_t(\psi_a)=\sharp\ \Delta_{\Fbar }(p, a),$$
where $\psi_a$ denotes the polynomial $\psi$ evaluated at $a$ in $\R[\varpi,b,t]$.
Let us define a Zariski open subset in $\Fbar^n$ as follows:
$$U_{\Fbar}:=\{a\in\Fbar^n\mid \deg_t\varphi = \deg\varphi_{p,a}(t) \text{ and } \deg_t\psi=\deg_t\psi_{a}\}.$$
Let $\mathcal U$ be a Zariski open subset in $\Fbar ^n$ obtained in Claim 4. Then $\mathcal U\cap U_{\Fbar}\neq \emptyset$,
and therefore 
\begin{align*}
 \sharp\ p\left(\Delta_{\Fbar }(p, a)\right)&=\deg_t(\varphi)=\deg\varphi_{p,a}(t)\\
 &= \deg_t(\psi)= \sharp\ \Delta_{\Fbar }(p, a)
\end{align*}
for all $a\in U_{\Fbar}$. Let $U$ be a Zariski open subset in $\R^n$ defined by
$$U:=\{a\in\R^n\mid \deg_t\varphi = \deg\varphi_{p,a}(t) \text{ and } \deg_t\psi=\deg\varphi_{g,a}\}.$$
This leads $U\subset U_{\Fbar}$, and therefore we conclude that all $a\in U$ are T-good for $p$.
\end{proof}

\begin{remark}
 \label{rmk:generic-x1} Let $\theta_{p,a}$ be the square-free part of a generator of $\J\cap \R(\varpi)[x_1]$, where
\begin{equation}\label{eq:J_ideal}
    \J:=\big\langle \Phi_1(x,a), \, \dots, \, \Phi_n(x,a), \, s\times\partial p/\partial x_n -1 \big\rangle
\end{equation}
is an ideal in $\R(\varpi)[s,x]$. Then there exists a Zariski open subset $U$ in $\R^n$ such that all elements $a$ in $U$ satisfy the condition $\deg \theta_{p,a} = \deg \varphi_{p,a}$. 

 Indeed, let us consider the point $\hat a$, its analytic neighbourhood $\mathcal U'$, and the analytic functions $\varphi_k\colon \mathcal U'\to \mathcal V_k$ as in the proof of Claim 3 of Theorem \ref{thm:generic-tangency}. Let $\mathcal U''$ be the set of elements $a$ in $\mathcal U'$ such that $\pi_1(\varphi_i(a))\neq \pi_1(\varphi_j(a))$ and $p(\varphi_i(a))\neq p(\varphi_j(a))$ for all $i\neq j $ in $\{1,\ldots,r\}$.  Here $\pi_1\colon \Fbar^{n}\to \Fbar$ denotes the projection on the first coordinate. Then for all $a\in \mathcal U''$ one has
 \begin{equation}\label{eq:theta_varphi}
 \deg \theta_{p,a} = \deg \varphi_{p,a}=\sharp\ \Delta_{\Fbar }(p, a).
\end{equation}
Applying the same argument as in the proof of Claim 5 of Theorem \ref{thm:generic-tangency} we obtain  a Zariski open subset $U$ in $\R^n$ such that all elements $a$ in $U$ satisfy \eqref{eq:theta_varphi}.
\end{remark}


From Theorem \ref{thm:generic-tangency} and Remark \ref{rmk:generic-x1}, we immediately obtain the following corollary.

\begin{corollary}\label{cor:generic-x1} There exists a Zariski open subset $U$ in $\R^n$ such that all elements $a$ in $U$ satisfy the condition \eqref{eq:theta_varphi}.
\end{corollary}

\subsection{A sufficient condition} Testing the equality \eqref{eq:def_Tgood} directly may be difficult in general. We therefore provide a  sufficient condition relying on \eqref{eq:theta_varphi} for the T-goodness of $a$ that can be verified algorithmically. 

Assume that $\Delta_{\Fbar}(p,a)$ is finite. Let  $G$ be the reduced Gr\"{o}bner basis of $\J$ given in \eqref{eq:J_ideal} w.r.t the pure lexicographic monomial order  $s>x_n>\cdots>x_1$ in $\R(\varpi)[s,x]$.

\begin{proposition}\label{prop:testTgood}
Assume that $\deg \theta_{p,a} = \deg \varphi_{p,a}$.
Then, the following statements are equivalent:
\begin{itemize}
    \item[(i)] $a$ is T-good for $p$;
    \item[(ii)] $\R(\varpi)[s,x]$ is in shape position w.r.t the order $>$, i.e.  there exist univariate polynomials $\theta_2,\dots,\theta_{n+1}$ in $\R(\varpi)[x_1]$ such that
\begin{equation}\label{eq:testTgood}
 G=\big[\theta_{p,a}(x_1), x_2-\theta_2(x_1),\dots, x_n-\theta_{n}(x_1), s-\theta_{n+1}(x_1)\big]. 
\end{equation}
\end{itemize}
 
\end{proposition}
\begin{proof} Clearly, $\Delta_{\Fbar }(p, a)$ is the image of $V_{\Fbar}(\J)$ under the natural projection \[ \Fbar^{n+1} \to \Fbar^n, \ (s,x)\mapsto x.\]
Hence, the $x_1$-coordinate values of the points in $\Delta_{\Fbar}(p,a)$ and in $V_{\Fbar}(\J)$ coincide.

Assume that $(\rm i)$ holds. By assumption, one has $\deg \theta_{p,a}=\sharp\Delta_{\Fbar }(p, a)$. Hence, the roots of $\theta_{p,a}(x_1)$ are the $x_1$-coordinate values of the points in $V_{\Fbar}(\J)$. According to the Shape Lemma \cite{gianni89}, $\R(\varpi)[s,x]$ is in shape position w.r.t the order $>$.

Suppose that $(\rm ii)$ holds. The roots of $\theta_{p,a}$ are the $x_1$-coordinate values of the points in $V_{\Fbar}(\J)$. 
It follows that the roots of $\theta(x_1)$ are the $x_1$-coordinate values of the points in $\Delta_{\Fbar }(p, a)$. Hence, 
\[\Delta_{\Fbar}(p, a)=\big\{x\in\Fbar^n:\theta_{p,a}(x_1)=0,x_2=\theta_2(x_1),\dots, x_n=\theta_{n}(x_1)
\big\},\]
and the cardinality of this set equals $\deg\theta_{p,a}$. Therefore, $a$ is T-good for $p$.
\end{proof}


Denote by $\mathcal{C}$ the following condition:
\[(\mathcal{C}) \quad \deg \theta_{p,a} = \deg \varphi_{p,a} \text{ and } \R(\varpi)[s,x] \text{ is in shape position}.\]

\begin{remark}\label{cor:C_condition} Proposition \ref{prop:testTgood} establishes that $(\mathcal{C})$ is sufficient to ensure that $a$ is T-good for $p$. In addition, Corollary \ref{cor:generic-x1} implies that this condition holds generically.
\end{remark}

 The polynomial $\theta_{p,a}$ is computed in the same manner as $\varphi$ in Step 2 of Algorithm \ref{alg:compute_phi_pa}, by working with the ideal $\J$. Once the reduced Gr\"{o}bner basis 
$G$ has been computed, checking whether it is in shape normal form \eqref{eq:testTgood} becomes a straightforward task.
 Hence, 
$(\mathcal{C})$ is checkable in practice. 
In the following, Algorithm \ref{alg:testTgood} is designed to check $(\mathcal{C})$.

\begin{algorithm}\caption{Testing $(\mathcal{C})$ - a sufficient condition for T-goodness}\label{alg:testTgood}
    \smallskip
\begin{flushleft}
    \textbf{Input:} $p$ in $\R[x]$ and $a\in\R^n$
\smallskip

    \textbf{Output:} $[\boo,\varphi]$ 
\end{flushleft}
    \begin{itemize}
\item [\rm 1:] Perform Algorithm \ref{alg:compute_phi_pa} with the input $p,a$, and let $\varphi$ be the output
\item [\rm 2:] Compute $G$ -- the reduced Gr\"{o}bner basis of $\J$ given in \eqref{eq:J_ideal} w.r.t the pure lexicographic monomial order  $s>x_n>\cdots>x_1$ in $\R(\varpi)[s,x]$
\item [\rm 3:] Let $\{\bar{\theta}\}=G \cap \R(\varpi)[t]$ 
and compute $\theta$ -- the square-free part of $\bar{\theta}$ 
\item [\rm 4:] If $(\mathcal{C})$ holds then return $[\texttt{true},\varphi]$, else return $[\texttt{false},\varphi]$
    \end{itemize}
\end{algorithm}

\textit{Description}.
The input of Algorithm \ref{alg:testTgood} is a polynomial
$p$ in $\R[x]$ together with a point $a\in\R^n$. Its output is the pair $[\boo,\varphi]$, where $\boo$ is a Boolean value that returns \texttt{true} if and only if $a$ is T-good for $p$, and $\varphi$ is the $n$-th non-critical tangency value polynomial of $p$ at $a$. The goal of Step 1 is computing $\varphi$ using Algorithm \ref{alg:compute_phi_pa}. Steps 2–3 are dedicated to computing $\theta$.


\textit{Correctness}. After Step 1, $\varphi$ is the $n$-th non-critical tangency value polynomial of $p$ at $a$. The polynomial $\theta$ obtained after Step~3 has the following property: the roots of $\theta(x_1)$ are the $x_1$-coordinate values of the points in $V_{\Fbar}(\J)$. One observes that condition $(\mathcal{C})$ holds if and only if, as stated in Proposition \ref{prop:testTgood}, $a$ is T-good for $p$. Equivalently, this occurs if and only if the Boolean variable \texttt{boo} is equal to \texttt{true}.

\begin{example}\label{ex:linear}
Consider the bivariate polynomial $p=x_1-x_2$ and $a=(1,3)$. After Steps 1 and 2, one has 
\[\varphi_{p,a}(t)=wt^{2}+4 wt + 4 w -2, \ \theta_{p,a}(x_1) = 2 wx_1^{2}-8 w x_1 +16 w -1.\]
Moreover, the $\J\cap\R(\varpi)[s,x_1,x_2]$ is in shape position the pure lexicographic monomial order  $s>x_2>x_1$ because of:
\[G=\big[2 wx_1^{2}-8 w x_1 +16 w -1, \ x_2+ x_1 - 4, \ s+1\big].\]
Hence, $(1,3)$ is T-good for $p$. Algorithm \ref{alg:testTgood} returns $[\texttt{true},wt^{2}+4 wt + 4 w -2]$.
\end{example}

\begin{example}\label{ex:xy_2} Consider the polynomial given in Example \ref{ex:xy_1} and $a=(0,0)$. From the results in the example, the  polynomial $\varphi_{p,a}$ is of degree $4$. It follows that 
$\sharp\ p\left(\Delta_{\Fbar }(p, a)\right)=4<8=\deg\theta_{p,a}\leq\sharp\ \Delta_{\Fbar }(p, a)$, hence $(0,0)$ is not T-good for $p$. However, one can verify by using Algorithm \ref{alg:testTgood} that $a=(1,3)$ is T-good for $p$. 
\end{example}

\section{Deciding lower-boundedness}\label{sec:lbounded}
In this section, we first establish a necessary and sufficient condition for the lower-boundedness of a polynomial. Based on this condition, we then develop a probabilistic algorithm to decide whether a given polynomial $p$ is lower-bounded.

Sturm's theorem will be applied to the interval $(-\infty_{\F},-\infty)$. For this purpose, we recall the definitions of the signs of a polynomial $\varphi\in\F[t]$ at $-\infty_{\F}$ and $-\infty$, which are stated below:
\begin{equation}\label{eq:sign_infty_F}
\mathrm{sign}(\varphi(-\infty_{\F}))=
\begin{cases}
\mathrm{sign}(\LC(\varphi)), \text{ if } \deg(\varphi) \text{ is even},\\
-\mathrm{sign}(\LC(\varphi)) \text{ if } \deg(\varphi) \text{ is odd},
\end{cases}
\end{equation}
and
 $\mathrm{sign}(\varphi(-\infty))=\mathrm{sign}(\varphi(-r))$, 
for all sufficiently large $r>0$. More precisely, if we write
 $\varphi(t)=\sum_{\alpha\geq \alpha_0} a_{\alpha}(t)\varpi^{\alpha}$
with $a_{\alpha}(t)\in \R[t]$ and $a_{\alpha_0}(t)\neq 0$, then 
\begin{equation}\label{eq:sign_infty}
\mathrm{sign}(\varphi(-\infty))=\mathrm{sign}(a_{\alpha_0}(-\infty))=
\begin{cases}
\mathrm{sign}(\LC(a_{\alpha_0})), \text{ if } \deg(a_{\alpha_0}) \text{ is even},\\
-\mathrm{sign}(\LC(a_{\alpha_0})) \text{ if } \deg(a_{\alpha_0}) \text{ is odd}.
\end{cases}
\end{equation}

\subsection{A necessary and sufficient condition}\label{subsec:iff_lbounded}
Let $p \in \R[x]$ and recall that $\varphi_{p,a}$ denotes the $n$-th non-critical tangency value polynomial of $p$ at $a \in \R^n$. With notations as in Section \ref{subsec:signchanges}, we define 
\begin{equation}\label{eq:vpa}
v(p,a):=v_{\varphi_{p,a}}(-\infty_{\F})-v_{\varphi_{p,a}}(-\infty).
\end{equation}
Once $\varphi_{p,a}$ is obtained, the numbers of sign variations of its Sturm sequence at $-\infty_{\F}$ and $-\infty$ can be computed using the rules in \eqref{eq:sign_infty_F} and \eqref{eq:sign_infty}.
Consequently, the invariant $v(p,a)$ is effectively computable. The next result provides a criterion for the lower-boundedness of a given polynomial. 

\begin{theorem}\label{thm:lower-bounded}
Let $p\in \R[x]$ be a 
polynomial and let $a\in\R^n$ be T-good for $p$. Then, the following are equivalent: 
\begin{itemize}
    \item[(i)] $p$ is lower-bounded;
    \item[(ii)] ${v}(p,a)=0$.
\end{itemize}
\end{theorem}
\begin{proof}
By Proposition \ref{prop:tangency_ord}, $p$ is lower-bounded if and only if, $\lim_{\varpi\to 0} p(x)$ belongs to $\R\cup\{+\infty\}$, equivalently $\mathrm{ord}\ p(x)\geq 0$, for all points $x $ in  $\Delta_{\F}(p, a)$. Combining Proposition \ref{prop:exist} and the equality \eqref{eq:important}, this is equivalent to the assertion that the polynomial $\varphi_{p,a}$ has no roots in the interval $(-\infty_{\F},-\infty)$. 
The theorem therefore follows from Sturm's theorem (Theorem \ref{sturmtheorem}).
\end{proof}

\begin{example}
Consider the bivariate polynomial $p=x_1-x_2$ and $a=(1,3)$. As pointed out in Example \ref{ex:linear}, $a$  is T-good for $p$ and 
$\varphi_{p,a}(t)=wt^{2}+4 wt + 4 w -2$. The Sturm sequence is given as follows
\[[wt^{2}+4 wt +4 w -2, 2 wt +4 w, 2].\]
The two lists of signs of the sequence at $-\infty_{\F}$ and $-\infty$ are $[+, -, +]$ and $[-, -, +]$, respectively. Hence, one has 
$ v_{\varphi_{p,a}}(-\infty_{\F})=2$ and $v_{\varphi_{p,a}}(-\infty)=1$, and $v{(p,a)}=1 $. According to Theorem \ref{thm:lower-bounded}, $p$ is not lower-bounded.
\end{example}

\subsection{An algorithm to decide lower-boundedness}\label{subsec:alg_lbounded} Based on the necessary and sufficient condition introduced in Theorem \ref{thm:lower-bounded}, we propose Algorithm \ref{alg:lbounded} to decide the lower-boundedness of an arbitrary real polynomial $p$.

Algorithm \ref{alg:lbounded} is probabilistic in nature, since it involves choosing a generic point $a \in \R^n$ to ensure that the T-goodness condition holds.

\textit{Description}.
The input of Algorithm \ref{alg:lbounded} is a polynomial
$p$ in $\R[x]$. The output is a Boolean value that returns \texttt{true} if and only if $p$ is lower-bounded.
In Step~1, we choose a random point $a\in\R^n$ and perform Algorithm \ref{alg:testTgood}. The resulting output $[\texttt{true},\varphi_{p,a}]$ confirms that $a$ is T-good for $p$ and provides the $n$-th non-critical tangency value polynomial  of $p$ at $a$. The goal of Step 2 is to determine the numbers of sign changes $-\infty_{\F}$ and $-\infty$ of the Sturm sequence of $\varphi_{p,a}$.
Finally, the value $v(p,a)$ is computed in Step 3. 

\textit{Correctness}. According to Corollary  \ref{cor:generic-x1} and Remark \ref{rmk:generic-x1}, the sufficient condition $\mathcal{C}$ of T-goodness is generic. Hence, if $a\in\R^n$ is chosen at random, then after Step 1, the first value of the output of Algorithm \ref{alg:testTgood} is $\texttt{true}$. The output of Algorithm \ref{alg:lbounded} turns \texttt{true} if and only if $v(p,a)=0$. By Theorem \ref{thm:lower-bounded}, this equality holds precisely if and only if 
$p$ is lower-bounded.

\begin{algorithm}\caption{Deciding the lower-boundedness of a polynomial $p$}\label{alg:lbounded}
    \smallskip
\begin{flushleft}
    \textbf{Input:} $p$ in $\R[x]$ 
\smallskip

    \textbf{Output:} a Boolean value which equals \texttt{true} if and only if $p$ is lower-bounded
\end{flushleft}
\begin{itemize}
    \item [\rm 1:] Take a random point $a\in\R^n$ and perform Algorithm \ref{alg:testTgood} with the output $[\texttt{true},\varphi_{p,a}]$ 

\item [\rm 2:] Compute the Sturm sequence of $\varphi_{p,a}$, and then the two numbers of sign changes $ v_{\varphi_{p,a}}(-\infty_{\F})$ and $v_{\varphi_{p,a}}(-\infty)$ relying on the rules \eqref{eq:sign_infty_F} and \eqref{eq:sign_infty}

\item [\rm 3:] Compute  $v(p,a)$ as in \eqref{eq:vpa}

\item [\rm 4:] Return $\texttt{true}$ if $v(p,a)=0$ and $\texttt{false}$ otherwise
    \end{itemize}
\end{algorithm}

\begin{example}\label{ex:xMotzkin1} Consider the bivariate Motzkin polynomial $m=x_1^4x_2^2 + x_1^2x_2^4 - 3x_1^2x_2^2 + 1$. One takes $a=(1,3)$ in Step 1 and can verify that $a$ is T-good for $m$. The corresponding tangency value polynomial $\varphi_{m,a}$, computed in Step 2, has degree $6$. The sequences of signs at $-\infty_{\F}$ and $-\infty$ of the Sturm sequence are identical, namely $[+, -, +, +, -, +, -]$. Consequently, $v{(m,a)}=0 $ and Algorithm \ref{alg:lbounded} returns $\texttt{true}$, and hence $m$ is lower-bounded.
\end{example}

\begin{example}\label{ex:xy_3}
Let $q$ be the polynomial considered in Examples \ref{ex:xy_1} and \ref{ex:xy_2}; recall that $q(x_1,x_2)=(x_1x_2-1)^2 + x_2^2$. We summarize the results of Algorithm~\ref{alg:lbounded} applied to $q$ and $q+x_1$ at the point $a=(1,3)$ in Table \ref{table:1}. The results indicate that $q$ is lower-bounded, whereas $q+x_1$ is not. 

\begin{table}[H]
\footnotesize
\begin{tabular}{|c|c|c|c|c|c|c|}
\hline
        & T-good & $\deg\varphi$ & signs at $-\infty_{\F}$       & signs at $-\infty$            & $v$ & output \\ \hline
$q$     & yes    & 8             & $[+, -, +, -, +, -, +, -, +]$ & $[+, -, +, -, +, -, +, -, +]$ & $0$      & \texttt{true}   \\ \hline
$q+x_1$ & yes    & 8             & $[+, -, +, -, +, -, +, -, +]$ & $[-, -, +, -, +, -, +, -, +]$ & $1$      & \texttt{false} \\ \hline
\end{tabular}

\caption{Results of Algorithm \ref{alg:lbounded} for $q $ and $ q+x_1$ at $a=(1,3)$}
\label{table:1}
\end{table}
\end{example}

\subsection{Applications to deciding non-negativity and convexity}\label{sec:apply}
This subsection applies the results in Section \ref{sec:lbounded} to 
decide non-negativity and convexity of polynomials. Recall from \cite{Blum1998,ahmadi2013np} that such decision problems are NP-hard in general.

\textit{Deciding non-negativity.} A real polynomial 
$p$ in $\R[x]$ is said to be non-negative over $\R^n$
if $p(x)$ for all $x\in\R^n$.
Characterizing non-negativity of polynomials is a fundamental problem in both real algebraic geometry and polynomial optimization. 

Assume that $p$ has degree $d$, and $p=p_{d}+\cdots+p_0$, where $p_k$ is homogeneous of degree $k$. Denote by $\bar{p}$ the homogenization of $p$,
$$\bar{p}(x,z):=p_d+zp_{d-1}+\cdots+z^d p_0\in \R[x,z].$$

If $p$ is homogeneous, then its lower boundedness and non-negativity are equivalent. This equivalence, however, does not hold for non-homogeneous polynomials. Nevertheless, by \cite[Lemma 3.3]{laurent2009} the non-negativity of $p$ is equivalent to the non-negativity of its homogenization $\bar{p}$. Hence, we obtain the following lemma:

\begin{lemma}
 \label{prop:nonneg} A polynomial $p\in \R[x]$ is non-negative over $\R^n$ if and only if $\bar p$ is lower-bounded over $\R^{n+1}$.
\end{lemma}



Thanks to Lemma \ref{prop:nonneg}, the decision problem for the non-negativity of an arbitrary polynomial 
$p$ can be transformed into that of determining the lower-boundedness of its homogenization 
$\bar{p}$. The latter can be addressed by Algorithm \ref{alg:lbounded} introduced in the preceding section. We illustrate the this method by considering the following simple two examples.

\begin{example}\label{ex:nonneg1} Consider the bivariate quadratic polynomial $p=x_1^2+x_2^2 - 3x_1x_2+1$. Its homogenization is a trivariate quadratic form $\bar p=x_1^2+x_2^2 - 3x_1x_2+x_3^2$; $a=(1,2,3)$ is T-good for $\bar p$. Algorithm \ref{alg:lbounded} returns \texttt{false} because of $v(\bar{p},a)=6-4\neq 0$. Hence, $p$ is not non-negative. 
\end{example}

\begin{example}\label{ex:nonneg2} Consider the bivariate quartic form $p= x_1^4 + x_2^4 + x_1^2x_2^2 - 3x_1^3x_2 - 3x_1x_2^3$. Since $p$ is homogeneous, its lower-boundedness and non-negativity are equivalent. Applying Algorithm \ref{alg:lbounded} to $p$ at the T-good point $a=(1,2)$, yields \texttt{false}, as $v(p,a)=4$. Therefore, $p$ is not non-negative.
\end{example}

\textit{Deciding convexity.}
The polynomial $p$ in $\R[x]$ is convex if the following inequality holds for all $\alpha,\beta \in \R^n$ and $\gamma \in [0,1]$:
$$p(\gamma\alpha + (1-\gamma) \beta) \leq \gamma p(\alpha) + (1-\gamma) p(\beta).$$

We write $\nabla^2p$ to indicate the Hessian matrix of $p$. For $k\in [n]$, there are $\binom{n}{k}$ principal minors of order $k$ of $\nabla^2p$. It is clear that if $p$ is in $\R[x]$, then so are the principal minors. Moreover, if $\deg p =d$ then the degree of each principal minor of order $k$ is at most $k\times (d-2)$; particularly, the degree of $\det (\nabla^2p)$ may reach $n(d-2)$.

From \cite[Proposition 1.2.6]{bertsekas2003convex}, $p$ is convex over $\R^n$ if and only if its Hessian matrix $\nabla^2p$ is positive semi-definite for any $x$ in $\R^n$. Thanks to 
\cite[Corollary 7.1.5 and Theorem 7.2.5]{horn2012matrix}, the latter conclusion holds if and only if the principal minors of $\nabla^2p$ are all non-negative over $\R^n$.
Therefore, a necessary and sufficient condition for the convexity of polynomials is given below:
    
 \begin{lemma}
 \label{propr:HessianPSD}
A polynomial $p\in \R[x]$ is convex over $\R^n$ if and only if $\bar{q}$ is lower-bounded, for all principal minor $q$ of $\nabla^2p$.
\end{lemma}

It follows from Lemma~\ref{propr:HessianPSD} that the convexity of $p$ can be determined by checking the lower-boundedness of the homogenizations of the principal minors of $\nabla^2p$. Recall that deciding lower-boundedness of polynomials can be decided using Algorithm \ref{alg:lbounded}. We we illustrate this procedure in the following example.

\begin{example}\label{ex:convex}
  Consider the bivariate quartic form $p= x_1^4 + x_2^4 + 10 x_1^2x_2^2$. The Hessian matrix of $p$ is given as follows:
 \[\nabla^2p=\begin{bmatrix}
12x_1^2 + 20x_2^2 & 40 x_1 x_2 \\
40x_1 x_2 & 20 x_1^2 + 12x_2^2
\end{bmatrix}
    \]
The principal minors (of order $1$ and $2$) of $\nabla^2p$ are the following polynomials:
    \[ 12x_1^2 + 20x_2^2, \ 20 x_1^2 + 12x_2^2, \
    240x_1^4 - 1056x_1^2x_2^2 + 240 x_2^4.\]
Since these are homogeneous polynomials, we apply Algorithm \ref{alg:lbounded}, which yields the results \texttt{true}, \texttt{true}, and \texttt{false}, respectively. Consequently, we conclude that 
$p$ is not convex.
    \end{example}

\section{Practical experiments}\label{sec:experiment}
This section is twofold. First, we present the experimental results obtained by applying Algorithm~\ref{alg:lbounded} to determine the lower-boundedness of polynomials. Second, we compare these results with those obtained using existing quantifier elimination. The algorithm was implemented in {\sc Maple} 2024, and the results were obtained on a MacBook Pro M3 with 18 GB of RAM. 

In Steps 1-2 of Algorithm~\ref{alg:lbounded}, to compute the reduced Gr\"{o}bner basis of an ideal with respect to a given order, we use the $\mathtt{Basis}$ command from the $\mathtt{Groebner}$ package. The square-free part of a polynomial $\R(\varpi)[t]$ is computed using the $\mathtt{SquareFreePart}$ command provided by the $\mathtt{PolynomialTools}$ package. Furthermore, we implement a procedure to construct the Sturm sequence of a polynomial $\varphi\in\R(\varpi)[t]$ and to determine the sign variation counts $ v_{\varphi}(-\infty_{\F})$ and $v_{\varphi}(-\infty)$ relying on the rules \eqref{eq:sign_infty_F} and \eqref{eq:sign_infty}.
The code used in our experiments is openly available at:
\texttt{https://github.com/hieuvut/lbound}.

Our experiment is conducted with the following family of dense bivariate polynomials of degree $d$ 
with coefficients $1$:
\[p_d:=\sum_{i+j\leq d}x_1^ix_2^j.\]
In Table \ref{table:2}, we ignore the subscripts $p,a$ since the context is clear.
Let $d$ range from $3$ to $6$. The values of $\deg \varphi_{p,a}$ are reported in the second column. The time of computing  $\varphi$, i.e. Step 2 in Algorithm \ref{alg:lbounded}, is reported in the third column. The reported timings correspond to real elapsed time measured in seconds. The second-to-last column lists the total running time of Algorithm \ref{alg:lbounded}, including the step that verifies whether $a=(1,3)$ is $T$-good for $p_d$. We were unable to obtain a result within four hours of computation, which is indicated by the symbol $-$. 

\begin{table}[H]
\begin{tabular}{|c|c|c|c|c|c|c|c|}
\hline
$d$ & $\deg \varphi$ & computing $\varphi$& $v_{\varphi}(-\infty_{\F})$ & $v_{\varphi}(-\infty)$ & output & Alg.\ref{alg:lbounded} & QE\\ \hline
$3$   & 4              & 0.05           & 3                  & 2                  & \texttt{false}  & 0.07 & 0.11 \\ \hline
$4$   & 8              & 0.73           & 6                  & 6                  & \texttt{true}   & 1.24 & 197 \\ \hline
$5$  & 10             & 5.72           & 8                  & 5                  & \texttt{false}  & 26.5 &  -- \\ \hline
$6$  & 12             & 100.1            & 10                 & 10                 & \texttt{true}  & 209.1 & -- \\ \hline
\end{tabular}


\caption{Results of Algorithm \ref{alg:lbounded} for $p_d$ at $a=(1,3)$}
\label{table:2}
\end{table}

The last column reports the computation time for solving \eqref{eq:QE} using quantifier elimination (QE) by employing the \texttt{QuantifierElimination} command from the \texttt{SemiAlgebraicSetTools} sub-package of the \texttt{RegularChains} package \cite{chen2016quantifier}, available in {\sc Maple}. Across the last three cases, Algorithm \ref{alg:lbounded} shows a substantial efficiency advantage compared relative to quantifier elimination. Nevertheless, certain special cases may still favor QE, particularly when the polynomial is sparse or highly structured. For example, for the polynomial $p = x_1^6 + x_2^6$, which is both sparse and homogeneous, Algorithm~\ref{alg:lbounded} required 2.56 seconds, whereas QE completed the task in only 0.04 seconds.



\section*{Conclusions and Discussion}

This paper introduced the non-critical tangency value polynomial of a real polynomial $p$ at a given real point $a$, along with the notion of T-goodness. By exploiting these concepts, we proposed a probabilistic algorithm that decides whether $p$ is lower-bounded. The experimental results indicate that the proposed algorithm performs efficiently on small-scale instances and, in particular, surpasses the quantifier elimination approach on dense polynomials.

This study has concentrated on establishing the theoretical guarantees and practical implementation of the proposed algorithm, while a detailed analysis of its computational complexity has been left beyond the present scope. It would be of significant interest to investigate whether the algorithm admits a single-exponential complexity in the number of variables 
$n$. Such an analysis would provide a deeper understanding of its scalability and could guide further improvements of the method.

\bibliographystyle{amsplain}

\bibliography{references.bib}
\end{document}